\documentclass[12pt,psfig,reqno]{amsart}
\usepackage{amsfonts}
\pdfoutput=1
\usepackage{amscd}
\usepackage{mathrsfs}
\usepackage{}
  \usepackage{hyperref}
\usepackage{amssymb,amsfonts,latexsym}
\usepackage{graphics,verbatim}
\usepackage{graphicx}
\usepackage{mathabx}
\usepackage[usenames]{color} 
\usepackage{cite}
\usepackage{verbatim}
\makeatletter

\renewcommand\section{\@startsection {section}{1}{\z@}%
                                   {-3.5ex \@plus -1ex \@minus -.2ex}%
                                   {2.3ex \@plus.2ex}%
                                   {\centering\normalfont\bf}}

 \numberwithin{equation}{section}
\numberwithin{equation}{section}

\newcommand{\I}{\int_{0}^{1}}      \newcommand{\Id}{\int_{\mathbb{D}}}
\newcommand{\sk}{\sum_{k=0}^{\infty}}
\newcommand{\sn}{\sum_{n=0}^{\infty}}              
 \newcommand{\bv}{\mathcal {B}_{\nu}}                     \newcommand{\ii}{\mathcal {I}}
  \newcommand{\bw}{\mathcal {B}_{\omega}}
\newcommand{\hd}{H(\mathbb{D})}            \newcommand{\rrp}{\mathcal {R}^{+}}
\newcommand{\dd}{\mathbb{D}}           \newcommand{\n}{\mathcal {N}}
    \newcommand{\B}{\mathcal {B}}
\newcommand{\w}{\omega}
\newcommand{\fbw}{||f||_{\mathcal {B}_{\omega}}}
\newcommand{\fbv}{||f||_{\mathcal {B}_{\nu}}}

\numberwithin{equation}{section}

\theoremstyle{plain}
\newtheorem{thm}{Theorem}[section]
\newtheorem{lemma}[thm]{Lemma}
\newtheorem{pro}[thm]{Proposition}
\newtheorem{cor}[thm]{Corollary}

\newtheorem{re}[thm]{Remark}
\newenvironment{proof of Theorem 1.3}{{\noindent\it Proof of Theorem 1.3}\quad}{\hfill $\square$\par}
\newenvironment{proof of Theorem 1.1}{{\noindent\it Proof of Theorem 1.1}\quad}{\hfill $\square$\par}
\begin{document}
\title{Generalized integral type Hilbert operator acting on weighted  Bloch  space}
\author{Pengcheng Tang$^{*}$ \ \ Xuejun  Zhang }

\address{ College of Mathematics and Statistics, Hunan Normal University, Changsha, Hunan 410006, China}

\email{1228928716@qq.com}

\date{}

\thanks{$^*$ Corresponding author.\\
 \ \ \ \ \ \ Email:www.tang-tpc.com@foxmail.com(Pengcheng Tang)\  xuejunttt@263.net(Xuejun Zhang)\\
\ \ \ \ \ \ The research is supported by the National Natural Science Foundation
of China (No.11942109) and Natural Science Foundation of Hunan Province of China (No. 2022JJ30369).
 }

\begin{abstract}
Let $\mu$ be a finite Borel measure on $[0,1)$. In this paper, we consider the generalized integral type Hilbert operator
 $$\ii_{\mu_{\alpha+1}}(f)(z)=\I \frac{f(t)}{(1-tz)^{\alpha+1}}d\mu(t)\ \ \ (\alpha>-1).$$
The operator $\ii_{\mu_{1}}$ has been extensively studied recently. The aim of this paper is to study the boundedness(resp. compactness) of    $ \ii_{\mu_{\alpha+1}}$ acting from the  normal weight Bloch  space into another of  the same kind. As consequences of our study, we
get completely results for the boundedness of  $ \ii_{\mu_{\alpha+1}}$ acting between  Bloch type spaces,  logarithmic Bloch
spaces among others.
\end{abstract}
\maketitle
\section{Introduction }\label{sec1}
\ \ \   Let $\mathbb{D}=\{z\in \mathbb{C}:\vert z\vert <1\}$ denote the open unit disk of the complex plane $\mathbb{C}$ and $H(\mathbb{D})$ denote the space of all analytic functions in $\mathbb{D}$.

 A positive continuous  function $\nu$ on  $[0,1)$ is called  normal
    if there exist  $0<a\leq b<\infty$ and $0\leq s_{0}<1$  such that
 $\displaystyle{\ \frac{\nu(s)}{(1-s^{2})^a}}$
   almost decreasing on $[s_{0},1)$ and
   $\displaystyle{\frac{\nu(s)}{(1-s^{2})^b}}$ almost increasing on $[s_{0},1)$.

 A function $g$  is almost increasing if there exists $C>0$ such that $r_{1}<r_{2}$ implies $g(r_{1})\leq Cg(r_{2})$. Almost decreasing functions are  defined in an analogous manner.

Functions such as
$$\nu(s)=(1-s^{2})^{t}\log^{\delta}\frac{e}{1-s^{2}}(t>0,\delta \in \mathbb{R})\  \mbox{and }\ \nu(s)=\left(\sum_{k=1}^{\infty}\frac{ks^{2k-2}}{\log^{3}(k+1)}\right)^{-1}$$
are normal functions.

In this paper,  we use $\mathcal {N}$ to denote the set of all normal functions  on $[0,1)$ and let $s_{0}=0$.  The letters $a$ and $b$ always represent the parameters in the definition of  normal function.

 Let $\nu \in \mathcal {N}$,   the normal weight Bloch space $\mathcal {B}_{\nu}$ consists of those functions $ f \in H(\mathbb{D})$ for which
$$||f||_{\mathcal {B}_{\nu}}=|f(0)|+\sup_{z\in \mathbb{D}}\nu(|z|)|f'(z)|<\infty.$$
In particular, if $\nu(|z|)=(1-|z|^{2})^{\gamma}(\gamma>0)$, then  $\bv$ is the Bloch type space $\mathcal {B}^{\gamma}$.  If $\nu(|z|)=(1-|z|^{2})\log^{-\beta}\frac{e}{1-|z|^{2}}(\beta \in \mathbb{R})$,  then $\bv$ is  just the  logarithmic Bloch space   $\B_{\log^{\beta}}$.

Let $\mu$  be a positive Borel measure on $[0, 1)$, $0\leq\gamma<\infty$ and $0 < s < \infty$. Then  $\mu$ is
a $\gamma$-logarithmic $s$-Carleson measure if there exists a positive constant $C$, such that
$$\mu([t,1))\log^{\gamma}\frac{e}{1-t} \leq C (1-t)^{s}, \ \ \mbox{for all}\  0\leq t<1 .$$
In particular,  $\mu$ is an  $s$-Carleson measure if $\gamma=0$. See \cite{H16} for more about logarithmic  Carleson measure.

Let $\mu$ be a finite Borel measure on $[0,1)$ and $n\in \mathbb{N}$. We use $\mu_{n}$ to denote the sequence of order $n$ of $\mu$, that is,  $\mu_{n}=\int_{[0,1)}t^{n}d\mu(t)$. Let $\mathcal {H}_{\mu}$ be the Hankel matrix $(\mu_{n,k})_{n,k\geq 0}$ with entries  $\mu_{n,k}=\mu_{n+k}$. The matrix $\mathcal {H}_{\mu}$ induces  an operator on $ H(\mathbb{D}) $ by its action on the Taylor coefficients $:a_{n}\rightarrow\displaystyle{\sum_{k=0}^{\infty}}\mu_{n,k}a_{k},\ \ n=0,1,2,\cdots $.
\vskip2mm
If $f(z)=\displaystyle{\sum_{n=0}^{\infty}}a_{n}z^{n} \in  H(\mathbb{D})$, the generalized Hilbert operator defined as follows:
$$
\mathcal {H}_{\mu}(f)(z)=\sum_{n=0}^{\infty}\left(\sum_{k=0}^{\infty}\mu_{n,k}a_{k}\right)z^{n},
$$
 It's known that the generalized Hilbert operator $\mathcal {H}_{\mu}$ is closely related to the integral operator
 $$\ii_{\mu}(f)(z)=\I \frac{f(t)}{1-tz}d\mu(t)$$
If  $\mu$ is the Lebesgue measure on $[0, 1)$, then  $\mathcal {H}_{\mu}$ and $ \ii_{\mu} $ reduce to the classic Hilbert operator $\mathcal {H}$ and $\ii$.

The action of the operators $\ii_{\mu}$ and $\mathcal {H}_{\mu}$ on distinct spaces of analytic functions
have been studied in a number of articles (see, e.g., \cite{H3,H11,H18,H21,H24,H22,H23})

 In this paper, we consider the generalized integral type Hilbert operator
 $$\ii_{\mu_{\alpha+1}}(f)(z)=\I \frac{f(t)}{(1-tz)^{\alpha+1}}d\mu(t),\ \ \ (\alpha>-1).$$
If $\alpha=0$, the operator $\ii_{\mu_{\alpha+1}}$ is just  $\ii_{\mu}$. The  integral type  operator $\ii_{\mu_{\alpha+1}}$ is closely related to the Hilbert type operator
$$
 \mathcal {H}^{\alpha}_{\mu}(f)(z)=\sum_{n=0}^{\infty}\frac{\Gamma(n+1+\alpha)}{\Gamma(n+1)\Gamma(\alpha+1)}\left(\sum_{k=0}^{\infty}\mu_{n,k}a_{k}\right)z^{n} , \ \ \ (\alpha>-1),
 $$
whenever the right hand side makes sense and defines an analytic function in $\mathbb{D}$.  $\mathcal {H}^{\alpha}_{\mu}$ can be regarded as the fractional derivative of  $\mathcal {H}_{\mu}$.  If $\alpha=1$, then $\mathcal {H}^{\alpha}_{\mu}$  called the Derivative-Hilbert operator which has been studied  in \cite{H19,d1} recently.

In \cite{H5}(see also \cite{H21}), the authors  have studied  the boundedness of  $\ii_{\mu}$ acting on $\mathcal {B}$. Li and Zhou  studied  the operator $\mathcal {H}_{\mu}$ from Bloch type spaces to the BMOA and the Bloch space in \cite{gi1}. Ye and Zhou  investigated $\ii_{\mu_{2}}$ acting  on $\mathcal {B}$ in \cite{H19}. The purpose of this article is to study the  boundedness(resp. compactness) of  $\ii_{\mu_{\alpha+1}}$ acting from   normal weight Bloch  space into another of  the same kind.  As consequences of our study, we
obtain several results about the boundedness of  $ \ii_{\mu_{\alpha+1}}$ acting between  Bloch type spaces,  logarithmic Bloch
spaces among others.

 Throughout the paper, the letter $C$ will denote an absolute constant whose value depends on the parameters
indicated in the parenthesis, and may change from one occurrence to another. We will use
the notation $``P\lesssim Q"$ if there exists a constant $C=C(\cdot) $ such that $`` P \leq CQ"$, and $`` P \gtrsim Q"$ is
understood in an analogous manner. In particular, if  $``P\lesssim Q"$  and $ ``P \gtrsim Q"$ , then we will write $``P\asymp Q"$.

\section{Preliminary Results}\label{sec2}

\ \ \  In \cite{i1}, a sequence $\{V_{n}\}$ was constructed in the following way:
Let $\psi$ be a $C^{\infty}$-function on $\mathbb{R}$ such that
(1)\ $\psi(s)=1$ for $s\leq 1$, (2)\ $\psi(s)=0 $ for $s\geq 2$,  (3)\ $\psi$ is decreasing and positive on the interval $(1,2)$.

Let $\varphi(s)=\psi(\frac{s}{2})-\psi(s)$, and let $v_{0}=1+z$, for $n\geq 1$,
$$V_{n}(z)=\sk \varphi(\frac{k}{2^{n-1}})z^{k}=\sum_{k=2^{n-1}}^{2^{n+1}-1}\varphi(\frac{k}{2^{n-1}})z^{k}.$$
The polynomials $V_{n}$ have the properties:
\\ (1)\ $\displaystyle{f(z)=\sn V_{n}\ast g (z)}$ , for $f\in \hd$;
\\ (2)\ $||V_{n}\ast f||_{p}\lesssim ||f||_{p}$, for $f\in H^{p}, p>0$;
\\ (3)\ $||V_{n}||_{p}\asymp 2^{n(1-\frac{1}{p})}$, for all $p>0$,
where $\ast$ denotes the Hadamard product and $||\cdot||_{p}$ denotes the norm of Hardy space $H^{p}$.

\begin{lemma}
Let $\nu \in \n$ and  $f\in H(\mathbb{D})$, then
$f\in \mathcal {B}_{\nu}$ if and only if  $$\sup_{n\geq0}\nu(1-2^{-n})2^{n}||V_{n}\ast f||_{\infty}<\infty.$$
Moreover, $$||f||_{\mathcal {B}_{\nu}}\asymp \sup_{n\geq0}\nu(1-2^{-n})2^{n}||V_{n}\ast f||_{\infty}.$$
\end{lemma}
The proof of this Lemma  is  similar to that Theorem  3.1 in \cite{i2},  we leave it to the readers.

\begin{lemma}\label{lem2.2}
Let $\nu\in \mathcal {N}$ and
$$
g(\zeta)=1+\sum_{s=1}^{\infty}2^{s}\zeta^{n_{s}} \ \  (\zeta \in \mathbb{D}),
$$
where $n_{s}$ is the integer part of $(1-r_{s})^{-1}$, $r_{0}=0$, $\nu(r_{s})=2^{-s}(s=1,2,\cdots)$. Then
 $g(r)$ is strictly increasing on $[0, 1)$ and there exist two positive constants $N_{1}$ and $N_{2}$ such that
$$
\inf_{[0,1)}\nu(r)g(r)=N_{1}>0, \ \ \sup_{\zeta\in \mathbb{D}}\nu(|\zeta|)|g(\zeta)|=N_{2}< +\infty.
$$
\end{lemma}
This result is originated from Theorem 1 in \cite{s13}.

\begin{lemma}\label{lem2.3}
If $\nu \in \mathcal {N}$, then
$$\frac{\nu(|z|)}{\nu(|w|)}\lesssim \left(\frac{1-|z|^{2}}{1-|w|^{2}}\right)^{a}+\left(\frac{1-|z|^{2}}{1-|w|^{2}}\right)^{b} \ \ \mbox{for all }\ z,w\in \dd.$$
\end{lemma}
This result comes  from Lemma 2.2 in \cite{L24}.

\begin{lemma}
Let $\nu \in \n$, $0<\delta<\frac{1}{e^{2}}$, then
$$\int_{e}^{\infty}\frac{e^{-\delta t}dt}{t\nu(1-\frac{1}{t})}\lesssim \frac{1}{\nu(1-\delta)}.$$
\end{lemma}
\begin{proof}
$$\int_{e}^{\infty}\frac{e^{-\delta t}dt}{t\nu(1-\frac{1}{t})}=\int_{e}^{\frac{1}{\delta}}\frac{e^{-\delta t}dt}{t\nu(1-\frac{1}{t})}+\int_{\frac{1}{\delta}}^{\infty}\frac{e^{-\delta t}dt}{t\nu(1-\frac{1}{t})}=I_{1}+I_{2}.$$
By the definition of normal function, we have
$$I_{1}\leq \int_{e}^{\frac{1}{\delta}} \frac{dt}{t\nu(1-\frac{1}{t})} \lesssim \frac{\delta^{a}}{\nu(1-\delta)} \int_{e}^{\frac{1}{\delta}}t^{a-1}dt
\lesssim  \frac{1}{\nu(1-\delta)}.$$
If $t>\frac{1}{\delta}$,  then $1-\frac{1}{t}>1-\delta$. The definition of normal function shows that
$$\frac{\nu(1-\delta)}{[1-(1-\delta)]^{b}} \lesssim \frac{\nu(1-\frac{1}{t})}{[1-(1-\frac{1}{t})]^{b}}.$$
Hence, we have
 \[ \begin{split}
I_{2}&=\int_{\frac{1}{\delta}}^{\infty}\frac{\nu(1-\delta)}{\nu(1-\frac{1}{t})}\frac{e^{-\delta t  }dt}{t\nu(1-\delta)}\\
&\lesssim \int_{\frac{1}{\delta}}^{\infty} \frac{\delta^{b}t^{b-1}e^{-\delta t}}{\nu(1-\delta)}dt =\frac{1}{\nu(1-\delta)}\int_{1}^{\infty}e^{-s}s^{b-1}ds\\
&\lesssim \frac{1}{\nu(1-\delta)}.
  \end{split} \]
  The proof is complete.
  \end{proof}

\begin{lemma}
Let $\mu$ be a positive Borel measure on $ [0, 1)$, $\beta>0$, $\gamma>0$. Let $\tau $ be the Borel measure on $[0, 1)$ defined by
$$d\tau(t)=\frac{d\mu(t)}{(1-t)^{\gamma}}.$$
Then, the following two conditions are equivalent.
\\(a)\ $\tau$ is a $\beta$-Carleson measure.
\\ (b)\ $\mu$ is a $\beta+\gamma$-Carleson measure.

\end{lemma}
\begin{proof}

 $(a) \Rightarrow (b)$. Assume (a).  Then there exists a positive constant $C>0$ such that
$$\int_{t}^{1}\frac{d\mu(r)}{(1-r)^{\gamma}}\leq C(1-t)^{\beta},\ \ t\in[0,1).$$
Using this and the fact that the function $x \rightarrow \frac{1}{(1-x)^{\gamma}}$ is increasing in $[0, 1)$, we
obtain
$$\frac{\mu([t,1))}{(1-t)^{\gamma}}\leq \int_{t}^{1}\frac{d\mu(r)}{(1-r)^{\gamma}}\leq C(1-t)^{\beta},\ \ t\in[0,1).$$
This shows that  $\mu$ is a $\beta+\gamma$-Carleson measure.

$(b) \Rightarrow (a)$. Assume (b). Then there exists a positive constant $C>0$ such that
$$\mu(t)\leq C(1-t)^{\beta+\gamma},\ \ t\in[0,1).$$
For $0<x<1$,
let $h(x)=\mu([0,x))-\mu([0,1))=-\mu([x,1))$. Integrating by
parts and using the inequality  above , we obtain
\[ \begin{split}
&\tau([t,1))=\int_{t}^{1}\frac{d\mu(x)}{(1-x)^{\gamma}}\\
&= \frac{1}{(1-t)^{\gamma}}\mu([t,1))-\lim_{x\rightarrow 1}\frac{1}{(1-x)^{\gamma}}\mu([x,1))+\gamma\int_{t}^{1}\frac{\mu([x,1))}{(1-x)^{\gamma+1}}dx\\
&= \frac{1}{(1-t)^{\gamma}}\mu([t,1))+\gamma\int_{t}^{1}\frac{\mu([x,1))}{(1-x)^{\gamma+1}}dx\\
& \lesssim (1-t)^{\beta}+\int_{t}^{1} (1-x)^{\beta-1}dx\lesssim  (1-t)^{\beta}.
  \end{split} \]
  Thus,  $\tau$ is an $\beta$-Carleson measure.
\end{proof}

\begin{lemma}
Let $\omega, \nu \in \mathcal {N}$. If $T$ is a bounded operator from $\bw$ into $\bv$, then $T$  is compact operator from $\bw$ into $\bv$ if and only if  for any bounded sequence $\{h_{n}\}$ in $\bw$ which converges to $0$ uniformly on every compact subset of $\mathbb{D}$, we have
$\lim_{n\rightarrow \infty}||T(h_{n})|| _{\bv}=0.$
\end{lemma}
The proof is similar to that of Proposition 3.11 in \cite{hb1}, we omit the details.

\section{Nonnegative Coefficients of  normal weight Bloch functions}

\ \ \  First,  we give a characterization of the functions   $f\in \hd $ whose sequence of Taylor coefficients is
non-negative which belongs to $\bv$.

\begin{thm}
Let  $\nu \in \mathcal {N}$ and $f\in \hd$, $f(z)=\sn a_{n}z^{n}$, $a_{n}\geq 0 $ for all $n\geq 0$. Then $f\in \bv$ if and only if
$$S(f):=\sup_{n\geq 1}\nu(1-\frac{1}{n})\sum_{k=1}^{n}ka_{k}<\infty.$$
Moreover,
$$\fbv \asymp S(f) + a_{0} .$$
\end{thm}
\begin{proof}
 \ If $f\in \bv $, then for each $n\in \mathbb{N}$,
  \[ \begin{split}
  \fbv &\geq  \sup_{z=1-\frac{1}{n}}\nu(|z|)|f'(z)|\\
 & \geq \nu(1-\frac{1}{n})\left|\sum_{k=1}^{\infty}ka_{k}(1-\frac{1}{n})^{k-1}\right|\\
  &\gtrsim \nu(1-\frac{1}{n})\sum_{k=1}^{n}ka_{k},
    \end{split} \]
   and hence $ S(f)\lesssim \fbv$.  Since $a_{0}=|f(0)|\leq \fbv$, we may obtain $$S(f)+a_{0}\lesssim \fbv.$$

    On the other hand, if   $ S(f)<\infty$, then
    $$\nu(1-2^{-j})\sum_{k=2^{j}}^{2^{j+1}-1}ka_{k}\lesssim S(f), \ \ j\in \mathbb{N}.$$
    For every  $z\in \dd$ and $\frac{1}{2}\leq|z| <1$, we have
      \[ \begin{split}
      |f'(z)|&=\left|\sum_{j=0}^{\infty}\sum_{k=2^{j}}^{2^{j+1}-1}ka_{k}z^{k-1}\right|
      \leq \sum_{j=0}^{\infty}\left(\sum_{k=2^{j}}^{2^{j+1}-1}ka_{k}|z|^{k-1}\right)\\
      &\lesssim  S(f)\sum_{j=0}^{\infty} \frac{|z|^{2^{j}}}{\nu(1-2^{-j})}.
        \end{split} \]
        Next, we show that
        $$ \sum_{j=0}^{\infty} \frac{|z|^{2^{j}}}{\nu(1-2^{-j})}\lesssim \frac{1}{\nu(|z|)}\ \ \mbox{for all}\ \frac{1}{2}\leq |z|<1.\eqno{(3.1)}$$
    For each $\frac{1}{2}\leq |z|=r<1$,
  by choosing  $m\geq 2$ such that $r_{m-1}\leq r \leq r_{m}$, where $r_{m}=1-2^{-m}$. Then
$$
\sum_{j=0}^{\infty}\nu^{-1}(1-2^{-j})r^{2^{j}}\leq \sum_{j=0}^{m}\nu^{-1}(1-2^{-j})+\sum_{j=m+1}^{\infty}\nu^{-1}(1-2^{-j})r^{2^{j}}=S_{1}+S_{2}.
$$
Using Lemma 2.3 we have
 \[ \begin{split}
 S_{1}&\lesssim \nu^{-1}(1-2^{-m})\sum_{j=0}^{m}\left((\frac{1}{2})^{(m-j)a}+(\frac{1}{\lambda})^{(m-j)b}\right)\\
 &\lesssim  \nu^{-1}(1-2^{-m}).
 \end{split} \]
On the other hand,
 \[ \begin{split}
 S_{2}&=\sum_{j=m+1}^{\infty}\nu^{-1}(1-2^{-j})r^{2^{j}} \leq \sum_{j=m+1}^{\infty}\nu^{-1}(1-2^{-j})r_{m}^{2^{m}\cdot2^{j-m}}\\
 &\leq \sum_{j=m+1}^{\infty}\nu^{-1}(1-2^{-j})e^{-2^{(j-m)}}=\sum_{l=1}^{\infty}\nu^{-1}(1-2^{-(l+m)})e^{-2^{l}}\\
 &\lesssim \nu^{-1}(1-2^{-m})\sum_{l=1}^{\infty}e^{-2^{ l}}2^{lb} \lesssim \nu^{-1}(1-2^{-m}).
 \end{split} \]
Note that $\nu^{-1}(1-2^{-m})\asymp \nu^{-1}(r)$,
 we see that (3.1) is valid for all $z\in \dd$.

        Therefore,
         $$|f(0)|+\sup_{z\in \dd}\nu(|z|)|f'(z)|\lesssim a_{0}+S(f).$$
         This proof is complete.
         \end{proof}
\begin{cor}
Let  $\gamma>0$ and $f\in \hd$, $f(z)=\sn a_{n}z^{n}$, $a_{n}\geq 0 $ for all $n\geq 0$. Then $f\in \mathcal {B}^{\gamma}$ if and only if
$$\sup_{n\geq 1}n^{-\gamma}\sum_{k=1}^{n}ka_{k}<\infty.$$
\end{cor}

If $f\in \bv$ has  nonnegative and non-increasing coefficients, then the result of Theorem 3.1 can be simpler.

\begin{thm}
Let   $f(z)=\sn a_{n}z^{n}\in \hd$ with $a_{n}$ nonnegative and non-increasing. Then $f\in \bv$ if and only if
$$ \sup_{n\geq 1}n^{2}\nu(1-\frac{1}{n})a_{n}<\infty.$$
Moreover,
$$\fbv\asymp  a_{0}+\sup_{n\geq 1}n^{2}\nu(1-\frac{1}{n})a_{n}.$$
\end{thm}
\begin{proof}
If $a_{n}$ nonnegative and non-increasing, then  $\sum_{k=1}^{n}ka_{k} \gtrsim n^{2}a_{n}$. The proof of the necessity follows from Theorem 3.1 immediately.

On the other hand, if $ M:=\sup_{n\geq 1}n^{2}\nu(1-\frac{1}{n})a_{n}<\infty$, then
$$a_{n}\lesssim \frac{M}{n^{2}\nu(1-\frac{1}{n})}  \ \  \mbox{for all}\  n \geq 1. $$
For every  $z\in \dd$ and  $\frac{1}{2}<|z|<1$,
$$|f'(z)|\leq \sum_{n=1}^{\infty}na_{n}|z|^{n-1}\lesssim  M\sum_{n=1}^{\infty} \frac{|z|^{n}}{n\nu(1-\frac{1}{n})}. $$
Let $$h_{x}(t)=\frac{x^{t}}{t\nu(1-\frac{1}{t})}\ \ \ x\in(0,1),$$
then $h_{x}$ is  decreasing in $t$, for sufficiently large $ t$ and each $x\in(0,1)$. So, by Lemma 2.4 we have
$$\sum_{n=1}^{\infty}\frac{|z|^{n}}{n\nu(1-\frac{1}{n})}\asymp \int_{e}^{\infty}\frac{e^{-t\log\frac{1}{|z|}}}{t\nu(1-\frac{1}{t})}dt\lesssim \frac{1}{\nu(1-\log\frac{1}{|z|})}\asymp \frac{1}{\nu(|z|)}.$$
This means that
$$\fbv \lesssim a_{0}+\sup_{n\geq 1}n^{2}\nu(1-\frac{1}{n})a_{n}.$$
The proof is complete.
\end{proof}
\begin{cor}
Let $\gamma>0$ and   $f(z)=\sn a_{n}z^{n}\in \hd$ with $a_{n}$ nonnegative and non-increasing. Then $f\in \mathcal {B}^{\gamma}$ if and only if
$$\sup_{n\geq 1}n^{2-\gamma}a_{n}<\infty.$$
\end{cor}

\section{Generalized integral type Hilbert operator acting on weighted Bloch space}

Let  $\omega \in  \n$, we write  $\widetilde{\omega}(t)=\int_{0}^{t}\frac{1}{\omega(s)}ds$.  We begin with  characterizing those  measure $\mu$ for which the operator $\ii_{\mu_{\alpha+1}}$ is well defined on $\bw$.
 \begin{pro}
 Let $\mu$ be a positive Borel measure on $[0, 1)$ and $\alpha>-1$. For any given $f\in \bw$, $\ii_{\mu_{\alpha+1}}(f)$ uniformly converges on any compact subset of $\dd$ if
and only if
$$\I (\widetilde{\omega}(t)+1 )d\mu(t)<\infty. \eqno{(4.1)}$$
 \end{pro}
  \begin{proof}
 Let $f\in \bw$, it is easy to verify that
  $$|f(z)|\lesssim (\widetilde{\omega}(|z|)+1)\fbw \ \ \mbox{for all}\ \  z\in \dd . \eqno{(4.2)}$$
  If (4.1) holds, then for each $0 < r < 1 $ and $z\in\dd$  with $|z| \leq r$, we have
 \[ \begin{split}
 |\ii_{\mu_{\alpha+1}}(f)(z)|& \leq \I \frac{|f(t)|}{|1-tz|^{\alpha+1}}d\mu(t)\\
 &\lesssim \frac{\fbw}{(1-r)^{\alpha+1}}\I (\widetilde{\omega}(t)+1 )d\mu(t)\\
 &\lesssim \frac{\fbw}{(1-r)^{\alpha+1}}.
  \end{split} \]
  This implies that   $\ii_{\mu_{\alpha+1}}(f)$ uniformly converges on any compact subset of $\dd$ and  analytic in $\dd$.

Suppose that the operator $\ii_{\mu_{\alpha+1}}$ is well defined in $\bw$. Considering the function
 $$f(z)=\int_{0}^{z}g(s)ds+1$$where $g$ is the function in Lemma 2.2 with  respect to $\omega$. Then Lemma 2.2 implies that  $f\in \bw$.
Since $\ii_{\mu_{\alpha+1}}(f)(z)$ is well
defined for every $z\in \dd$,  we have
$$|I_{\mu_{\alpha+1}}(f)(0)|=\left|\I f(t)d\mu(t)\right|<\infty.$$
Note that  $\mu $ is a  positive  measure and  $g(s)>0$ for all $s\in [0,1)$. By Lemma 2.2, we  obtain that
$$f(t)=\int_{0}^{t}g(s)ds+1\asymp \widetilde{\omega}(t)+1 .
\eqno{(4.3)}$$
Therefore,
$$\I (\widetilde{\omega}(t)+1 )d\mu(t)<\infty.  $$
This proof is complete.
 \end{proof}

The sublinear  generalized integral type Hilbert operator $\widetilde{\ii}_{\mu_{\alpha+1}}$ defined by
$$\widetilde{\ii}_{\mu_{\alpha+1}}(f)(z)=\I \frac{|f(t)|}{(1-tz)^{\alpha+1}}d\mu(t), \ \ \ (\alpha>-1).$$
Obviously, Proposition 3.3 is also valid if   $\ii_{\mu_{\alpha+1}}$  is replaced by  $\widetilde{\ii}_{\mu_{\alpha+1}}$.
By mean of  Lemma 2.1, Theorem 3.1 and  the sublinear integral type Hilbert operator $\widetilde{\ii}_{\mu_{\alpha+1}}$, we have the following results.

 \begin{thm}
 Let $\omega,\nu \in \mathcal {N}$ and $\alpha>-1$. Suppose that $\mu$ is a positive Borel measure on $[0, 1)$  and satisfies  (4.1). Then the following statements are equivalent.
 \\ (a) $\ii_{\mu_{\alpha+1}}: \bw\rightarrow \bv$ is bounded;
 \\ (b) $\widetilde{\ii}_{\mu_{\alpha+1}}: \bw\rightarrow \bv$ is bounded;
 \\ (c) $\displaystyle{\sup_{n\geq 1}n^{\alpha+2}\nu(1-\frac{1}{n})\I t^{n}(\widetilde{\omega}(t)+1)d\mu(t)<\infty.}$
 \end{thm}
 \begin{proof}
 $(a)\Rightarrow(c):$ If $\ii_{\mu_{\alpha+1}}: \bw\rightarrow \bv$ is bounded.
 For  each $f\in\bw$,  Proposition 3.3 implies that $\ii_{\mu_{\alpha+1}}(f)$ converges absolutely
 for every $z\in \dd$ and
 $$\ii_{\mu_{\alpha+1}}(f)(z)=\sum_{n=0}^{\infty}\left(\frac{\Gamma(n+1+\alpha)}{\Gamma(n+1)\Gamma(\alpha+1)}\I t^{n}f(t)d\mu(t)\right)z^{n},\ \ z\in\dd. $$
 Take
 $$ f(z)=\int_{0}^{z}g(s)ds+1,$$
 where $g$ is the function in Lemma 2.2 with  respect to $\omega$.
Then $f\in \bw$ and
$$
\ii_{\mu_{\alpha+1}}(f)(z)=\I \frac{f(t)}{(1-tz)^{\alpha+1}}d\mu(t)=\sum_{n=0}^{\infty}b_{n}z^{n}
$$
where
  $$b_{n}=\frac{\Gamma(n+1+\alpha)}{\Gamma(n+1)\Gamma(\alpha+1)}\I t^{n}\left(\int_{0}^{t}g(s)ds+1\right)d\mu(t).$$
  It is clear that  $\{b_{n}\}_{n=1}^{\infty}$ is a  nonnegative sequence.
Using Theorem 3.1,  (4.3)   and Stirling's formula   we have
  \[ \begin{split}
  ||\ii_{\mu_{\alpha+1}}(f) ||_{\bv}&\gtrsim  \sup_{n\geq 1}\nu(1-\frac{1}{n})\sum_{k=1}^{n}kb_{k}\\
  &\gtrsim \sup_{n\geq 1}\nu(1-\frac{1}{n})\I t^{n}\left(\widetilde{\omega}(t)+1\right)d\mu(t)\sum_{k=1}^{n} k^{\alpha+1}\\
 & \asymp \sup_{n\geq 1}n^{\alpha+2}\nu(1-\frac{1}{n})\I t^{n}(\widetilde{\omega}(t)+1)d\mu(t).
   \end{split} \]
   Thus, $$\sup_{n\geq 1}n^{\alpha+2}\nu(1-\frac{1}{n})\I t^{n}(\widetilde{\omega}(t)+1)d\mu(t)<\infty.$$
   $(c)\Rightarrow(b):$ Assume (c). Then for each $n\in \mathbb{N}$, we have
    $$\I t^{n}(\widetilde{\omega}(t)+1)d\mu(t)\lesssim \frac{1}{n^{\alpha+2}\nu(1-\frac{1}{n})} . \eqno{(4.4)}$$
 For a given  $0 \not\equiv f \in \bw$,
 $$
\widetilde{\ii}_{\mu_{\alpha+1}}(f)(z)=\I \frac{|f(t)|}{(1-tz)^{\alpha+1}}d\mu(t)=\sum_{n=0}^{\infty}c_{n}z^{n},
 $$
 where  $$c_{n}=\frac{\Gamma(n+1+\alpha)}{\Gamma(n+1)\Gamma(\alpha+1)}\I t^{n}|f(t)|d\mu(t).$$
Obviously, $\{c_{n}\}_{n=1}^{\infty}$ is   a  nonnegative sequence. Using  (4.2),   (4.4), and the definition of normal weight,  we deduce that
    \[ \begin{split}
     & \ \ \ \ |c_{0}|+ \sup_{n\geq 1}\nu(1-\frac{1}{n})\sum_{k=1}^{n} kc_{k}\\
    &\lesssim  \fbw+\fbw \sup_{n\geq 1} \nu(1-\frac{1}{n}) \sum_{k=1}^{n}(k+1)^{\alpha+1}\I t^{k} (\widetilde{\omega}(t)+1)d\mu(t)\\
    &\lesssim  \fbw+\fbw\sup_{n\geq 1}\nu(1-\frac{1}{n})  \sum_{k=1}^{n} \frac{1}{k\nu(1-\frac{1}{k})}\\
     &\lesssim \fbw+\fbw
     \sup_{n\geq 1}\frac{1}{(n+1)^{a}} \sum_{k=1}^{n}(k+1)^{a-1} \\
   &  \lesssim \fbw.
     \end{split} \]
     Now, Theorem 3.1 shows that $\widetilde{\ii}_{\mu_{\alpha+1}}(f) \in \bv$ and $\widetilde{\ii}_{\mu_{\alpha+1}}: \bw\rightarrow \bv$ is bounded.

     $(b)\Rightarrow (a):$ If  $\widetilde{\ii}_{\mu_{\alpha+1}}: \bw\rightarrow \bv$ is bounded, for each $f\in \bw$,  by Lemma 2.1 we get
     $$\sup_{n\geq 1}\nu(1-2^{-n})2^{n}||V_{n}\ast \widetilde{\ii}_{\mu_{\alpha+1}}(f)||_{\infty}\asymp ||\widetilde{\ii}_{\mu_{\alpha+1}}(f)||_{\bv}\lesssim \fbw ||\widetilde{\ii}_{\mu_{\alpha+1}}||.$$
     Since the coefficients of $\widetilde{\ii}_{\mu_{\alpha+1}}(f)$ are non-negative, it is easy to check that
     $$M_{\infty}(r,V_{n}\ast \ii_{\mu_{\alpha+1}}(f) )\leq M_{\infty}(r,V_{n}\ast \widetilde{\ii}_{\mu_{\alpha+1}}(f) ) \ \
     \mbox{for all}\  0<r<1 .$$
     Therefore, $$||V_{n}\ast \ii_{\mu_{\alpha+1}}(f)||_{\infty}=\sup_{0<r<1}M_{\infty}(r,V_{n}\ast \ii_{\mu_{\alpha+1}}(f) )\leq  ||V_{n}\ast \widetilde{\ii}_{\mu_{\alpha+1}}(f)||_{\infty}.$$
Consequently,
$$||\ii_{\mu_{\alpha+1}}(f)||_{\bv}\asymp\sup_{n\geq 1}\nu(1-2^{-n})2^{n}||V_{n}\ast \ii_{\mu_{\alpha+1}}(f)||_{\infty} \lesssim \fbw.$$
This implies that
 $\ii_{\mu_{\alpha+1}}: \bw\rightarrow \bv$ is bounded.
 \end{proof}

\begin{thm}
 Let $\omega,\nu \in \mathcal {N}$ and $\alpha>-1$.  Suppose that $\mu$ is a finite positive Borel measure on $[0, 1)$  and $\widetilde{\omega}(1)<\infty$. Then the following statements are equivalent.
 \\ (a) \ $\ii_{\mu_{\alpha+1}}: \bw\rightarrow \bv$ is bounded;
\\ (b) \  $\widetilde{\ii}_{\mu_{\alpha+1}}: \bw\rightarrow \bv$ is bounded;
 \\ (c) \ $\ii_{\mu_{\alpha+1}}: \bw\rightarrow \bv$ is compact;
\\ (d) \  $\widetilde{\ii}_{\mu_{\alpha+1}}: \bw\rightarrow \bv$ is compact;
 \\ (e) \ $\displaystyle{\sup_{n\geq 1}n^{\alpha+2}\nu(1-\frac{1}{n})\mu_{n}<\infty.}$
\end{thm}
\begin{proof}
The  equivalence of $(a)\Leftrightarrow (b)\Leftrightarrow (e)$ follows from Theorem 4.2 immediately and  $(d)\Rightarrow (c)\Rightarrow (a)$ are obvious.   Therefore,    we only need to prove that $(e)\Rightarrow (d)$.

 Let $\{f_{k}\}_{k=1}^{\infty}$ be a bounded
sequence in $\bw$  which converges to $0$ uniformly on every compact subset of  $\mathbb{D}$.  Since $\widetilde{\omega}(1)<\infty$,  arguing as  the proof of Lemma 2.5 in \cite{Y1}, we have that
$$\lim_{k\rightarrow \infty}\sup_{z\in \dd}|f_{k}(z)|=0.$$
For each  $k\in \mathbb{N}$, we have
$$
\widetilde{\ii}_{\mu_{\alpha+1}}(f_{k})(z)=\I \frac{|f_{k}(t)|}{(1-tz)^{\alpha+1}}d\mu(t)=\sum_{n=0}^{\infty}c_{n,k}z^{n},
 $$
 where  $$c_{n,k}=\frac{\Gamma(n+1+\alpha)}{\Gamma(n+1)\Gamma(\alpha+1)}\I t^{n}|f_{k}(t)|d\mu(t).$$
It is obvious that $\{c_{n,k}\}_{n=1}^{\infty}$ is a nonnegative sequence for each $k \in \mathbb{N}$. To prove that $\widetilde{\ii}_{\mu_{\alpha+1}}: \bw\rightarrow \bv$ is compact, it is sufficient  to prove that
$$\lim_{k\rightarrow \infty}\left(c_{0,k}+\sup_{n\geq 1}\nu(1-\frac{1}{n})\sum_{j=1}^{n}jc_{j,k}\right) =0$$
by using Theorem 3.1  and Lemma 2.6.
If  $\displaystyle{\sup_{n\geq 1}n^{\alpha+2}\nu(1-\frac{1}{n})\mu_{n}<\infty}$, then
$$\mu_{n}\lesssim \frac{1}{n^{\alpha+2}\nu(1-\frac{1}{n})}\ \mbox{for all}\ n\in \mathbb{N}.$$
By Stirling's formula and the above inequality, we have
       \[ \begin{split}
& \ \ \ \ \  |c_{0,k}|+\sup_{n\geq 1}\nu(1-\frac{1}{n})\sum_{j=1}^{n}jc_{j,k}\\
& \lesssim \I |f_{k}(t)|d\mu(t)+ \sup_{n\geq 1}\nu(1-\frac{1}{n})\sum_{j=1}^{n}j^{\alpha+1}\I t^{j}|f_{k}(t)|d\mu(t)\\
&\lesssim \sup_{t\in [0,1)}|f_{k}(t)|+\sup_{t\in [0,1)}|f_{k}(t)|  \sup_{n\geq 1}\nu(1-\frac{1}{n})\sum_{j=1}^{n}j^{\alpha+1}\mu_{j}\\
&\lesssim \sup_{t\in [0,1)}|f_{k}(t)|+\sup_{t\in [0,1)}|f_{k}(t)|  \sup_{n\geq 1}\nu(1-\frac{1}{n})\sum_{j=1}^{n}\frac{1}{j\nu(1-\frac{1}{j})}\\
&\lesssim \sup_{t\in [0,1)}|f_{k}(t)|\rightarrow 0 , \ \ (k\rightarrow\infty).
         \end{split} \]
Hence (d) holds.
\end{proof}

\section{Some Applications}

As a direct application of the above results, we first consider the operator $\ii_{\mu_{\alpha+1}}$ acting from $\B^{\beta}$ to $\B^{\gamma}$. If $\gamma\geq \alpha+2$, then it is easy to see that  $\ii_{\mu_{\alpha+1}}: \mathcal {B}^{\beta}\rightarrow\mathcal {B}^{\gamma}$ is always a  bounded operator under the condition (4.1). Therefore,  we only need to consider  $0<\gamma<\alpha+2$.

\begin{thm}
Let $\mu$ be a positive Borel measure on $[0, 1)$ and satisfies  $\I \log\frac{e}{1-t}d\mu(t)<\infty$, $\alpha>-1$.  If $0<\gamma<\alpha+2$, then  the following statements are equivalent.
\\ (a)\ $\ii_{\mu_{\alpha+1}}: \mathcal {B}\rightarrow\mathcal {B}^{\gamma}$ is bounded;
\\(b)\ $\mu$ is a $1$-logarithmic $\alpha+2-\gamma$-Carleson measure;
\\(c)\ $\displaystyle{\I t^{n}\log\frac{e}{1-t}d\mu(t)=O(\frac{1}{n^{\alpha+2-\gamma}})}$.
\end{thm}
\begin{proof}
Let $d\lambda(t)=\log\frac{e}{1-t}d\mu(t)$, then Lemma 2.5 in  \cite{H5} shows that
$\mu$ is a $1$-logarithmic
$\alpha+2-\gamma$-Carleson measure if and only if $\lambda$ is an $\alpha+2-\gamma$-Carleson measure. By Theorem 2.1 in \cite{H10},  $\lambda$ is an $\alpha+2-\gamma$-Carleson measure if and only if $$\I  t^{n}d\lambda(t)= O(\frac{1}{n^{\alpha+2-\gamma}}).$$
The desired result follows  from Theorem 4.2 immediately.
\end{proof}
\begin{re}
If $\gamma=1$ and $\alpha=0$, the result of Theorem 5.1  have been obtained in  \cite{H5}(or \cite{H21}).   In addition, if  $\gamma=1$ and  $\alpha=1$, the result have been given in \cite{H19}.
\end{re}

\begin{thm}
Let $\mu$ be a positive Borel measure on $[0, 1)$ and satisfies  $\I \frac{d\mu(t)}{(1-t)^{\beta-1}}<\infty$,  $\alpha>-1$. If $0<\gamma<\alpha+2$ and $\beta>1$, then $\ii_{\mu_{\alpha+1}}: \mathcal {B}^{\beta}\rightarrow\mathcal {B}^{\gamma}$ is bounded if and only if $\mu$ is an $\alpha+1+\beta-\gamma$-Carleson measure.
\end{thm}
\begin{proof}
It follows from  Theorem 4.2  that $\ii_{\mu_{\alpha+1}}: \mathcal {B}^{\beta}\rightarrow\mathcal {B}^{\gamma}$ is bounded if and only if
 $$\I t^{n}\frac{d\mu(t)}{(1-t)^{\beta-1}}= O(\frac{1}{n^{\alpha+2-\gamma}}).$$
This  is equivalent to  saying that $\frac{d\mu(t)}{(1-t)^{\beta-1}}$ is an $\alpha+2-\gamma$-Carleson measure. The proof can be done by  using  Lemma 2.5.
\end{proof}

\begin{thm}
Let $\mu$ be a finite  positive Borel measure on $[0, 1)$ and  $\alpha>-1$. If $0<\gamma<\alpha+2$ and $0<\beta<1$, then the following statements are equivalent.
 \\ (a)\ $\ii_{\mu_{\alpha+1}}: \mathcal {B}^{\beta}\rightarrow\mathcal {B}^{\gamma}$ is bounded;
 \\ (b)\ $\ii_{\mu_{\alpha+1}}: \mathcal {B}^{\beta}\rightarrow\mathcal {B}^{\gamma}$ is compact;
 \\ (c)\  $\mu$ is an $\alpha+2-\gamma$-Carleson measure.
\end{thm}
\begin{proof}
This is a direct consequence of Theorem 4.3.
\end{proof}
\begin{re}
It should be mentioned that Ye and Zhou have  obtained some results of Theorem 5.1-5.4 by using the duality theorem in \cite{Y}. In fact, they  dealt with $\gamma=\alpha$ and  $\alpha\geq 1$.
\end{re}

In what follows, we consider the operator $\ii_{\mu_{\alpha+1}}$ acting between  logarithmic Bloch spaces.
\begin{thm}
Let $\alpha>-1$, $\beta>-1$, $\gamma\in \mathbb{R}$. Suppose that $\mu$ is a positive Borel measure on $[0, 1)$ and satisfies
$\I \frac{\log^{\beta}\frac{e}{1-t}}{1-t}d\mu(t)<\infty.$
Then  the following statements are equivalent.
\\(a)\ $\ii_{\mu_{\alpha+1}}: \mathcal {B}_{\log^{\beta}}\rightarrow\mathcal {B}_{\log^{\gamma}}$ is bounded;
\\ (b)\ $\displaystyle{\sup_{n\geq1}n^{\alpha+1}\log^{-\gamma}(n+1)\I t^{n} \log^{\beta+1}\frac{e}{1-t}d\mu(t)<\infty}$;
\\(c)\  $\displaystyle{\sup_{t\in [0,1)} \frac{\mu([t,1))(\log\frac{e}{1-t})^{\beta+1-\gamma}}{(1-t)^{\alpha+1}}<\infty  }$.
\end{thm}
\begin{proof}
It follows from Theorem 4.2 that $(a)\Leftrightarrow (b)$. We only need to show  that $(b)\Leftrightarrow (c)$.
The implication $(b)\Rightarrow(c)$ follows from the inequality
$$\mu\left([1-\frac{1}{n},1)\right)\log^{\beta+1}(n+1)\lesssim \int_{1-\frac{1}{n}}^{1} t^{n}\log^{\beta+1}\frac{e}{1-t}d\mu(t)\lesssim \frac{\log^{\gamma}(n+1)}{n^{\alpha+1}}.$$

 $(c)\Rightarrow(b)$.    Assume (c). Then   there exists a positive constant $C$ such that
 $$\mu\left([t,1)\right)\left(\log\frac{e}{1-t}\right)^{\beta+1-\gamma}\leq C (1-t)^{\alpha+1}, \ \ 0\leq t<1.$$
 Integrating by parts, we obtain
  \[ \begin{split}
 &\ \ \ \   \ \   \I t^{n}\log^{\beta+1}\frac{e}{1-t}d\mu(t)\\
 &=n \I t^{n-1}\mu([t,1))\log^{\beta+1}\frac{e}{1-t}dt+(\beta+1)\I t^{n}\mu([t,1))\log^{\beta}\frac{e}{1-t}\frac{dt}{1-t}\\
& \lesssim n \I t^{n-1}(1-t)^{\alpha+1}\log^{\gamma}\frac{e}{1-t}dt+ \I t^{n}(1-t)^{\alpha}\log^{\gamma-1}\frac{e}{1-t}dt.
     \end{split} \]
     Note that
    $$ \phi_{1}(t)=(1-t)^{\alpha+1}\log^{\gamma}\frac{e}{1-t}, \ \  \phi_{2}(t)= (1-t)^{\alpha}\log^{\gamma-1}\frac{e}{1-t} $$
    are regular in the sense of  \cite{L27}. Then, using
Lemma 1.3 and (1.1) in \cite{L27}, we have
$$ n \I t^{n-1}(1-t)^{\alpha+1}\log^{\gamma}\frac{e}{1-t}dt \asymp \frac{\log^{\gamma}(n+1)}{n^{\alpha+1}} $$
and
$$\I t^{n}(1-t)^{\alpha}\log^{\gamma-1}\frac{e}{1-t}dt\asymp \frac{\log^{\gamma-1}(n+1)}{n^{\alpha+1}}.$$
These two estimates imply that
$$ \I t^{n}\log^{\beta+1}\frac{2}{1-t}d\mu(t)\lesssim \frac{\log^{\gamma}(n+1)}{n^{\alpha+1}}. $$
Thus,  (b) holds.
\end{proof}
Arguing as  the proof of previous theorem, one can obtain the following theorems.
\begin{thm}
Let $\alpha>-1$, $\beta=-1$, $\gamma\in \mathbb{R}$. Suppose that $\mu$ is a positive Borel measure on $[0, 1)$ and satisfies
$\I \log\log\frac{e}{1-t}d\mu(t)<\infty$.
Then  the following statements are equivalent.
\\(a)\ $\ii_{\mu_{\alpha+1}}: \mathcal {B}_{\log^{-1}}\rightarrow\mathcal {B}_{\log^{\gamma}}$ is bounded;
\\ (b)\ $\displaystyle{\sup_{n\geq1}n^{\alpha+1}\log^{-\gamma}(n+1)\I t^{n} \log\log\frac{e}{1-t}d\mu(t)<\infty}$;
\\(c)\  $\displaystyle{\sup_{t\in [0,1)} \frac{\mu([t,1))\log\log\frac{e}{1-t}}{(1-t)^{\alpha+1}\log^{\gamma}\frac{e}{1-t}}<\infty  }$.
\end{thm}

\begin{thm}
Let $\alpha>-1$, $\beta<-1$, $\gamma\in \mathbb{R}$. Suppose that $\mu$ is a finite  positive Borel measure on $[0, 1)$, then  the following statements are equivalent.
\\(a)\ $\ii_{\mu_{\alpha+1}}: \mathcal {B}_{\log^{\beta}}\rightarrow\mathcal {B}_{\log^{\gamma}}$ is bounded;
\\(b)\  $\ii_{\mu_{\alpha+1}}: \mathcal {B}_{\log^{\beta}}\rightarrow\mathcal {B}_{\log^{\gamma}}$ is compact;
\\ (c)\ $\displaystyle{\sup_{n\geq1}n^{\alpha+1}\log^{-\gamma}(n+1)\mu_{n}<\infty}$;
\\(d)\  $\displaystyle{\sup_{t\in [0,1)} \frac{\mu([t,1))\log^{-\gamma}\frac{e}{1-t}}{(1-t)^{\alpha+1}}<\infty  }$.
\end{thm}

It is known that $\mathcal {H}$ maps $\mathcal {B}_{\log^{\beta}}$ into  $\mathcal {B}_{\log^{\beta+1}}$ for all $\beta \in  \mathbb{R}$(see e.g., \cite{i3}).  If   $\mu$ is Lebesgue measure on $[0, 1)$, then  Theorem 5.6-5.8  show that  the integral type Hilbert operator $\ii: \mathcal {B}_{\log^{\beta}}\rightarrow\mathcal {B}_{\log^{\beta+1}}$ is bounded if and only if $\beta>-1$.

%

\subsection*{Declarations}
The authors declare that there are no conflicts of interest regarding the publication of this paper.

  \subsection*{Availability of data and material}
Data sharing not applicable to this article as no datasets were generated or analysed during
the current study: the article describes entirely theoretical research.

\bibliographystyle{els}
\bibliography{sn-bibliography}

\end{document}